\documentclass[11pt]{amsart}
 
\usepackage[T1]{fontenc}
\usepackage[english]{babel}
\usepackage[margin=2cm]{geometry}
\usepackage{amsmath,amssymb,amsthm,mathtools}
\usepackage[shortlabels]{enumitem}
\usepackage{aliascnt}
\usepackage[colorlinks=true,citecolor=red,linkcolor=blue]{hyperref}
\usepackage[nameinlink,noabbrev]{cleveref}
\usepackage{xcolor}
 
\theoremstyle{plain}

\newaliascnt{proposition}{theorem}
\newtheorem{proposition}[proposition]{Proposition}
\aliascntresetthe{proposition}
 
\newaliascnt{corollary}{theorem}
\newtheorem{corollary}[corollary]{Corollary}
\aliascntresetthe{corollary}
 
\newaliascnt{lemma}{theorem}
\newtheorem{lemma}[lemma]{Lemma}
\aliascntresetthe{lemma}
 
\theoremstyle{remark}
\newaliascnt{remark}{theorem}
\newtheorem{remark}[remark]{Remark}
\aliascntresetthe{remark}
 
\theoremstyle{definition}
\newaliascnt{example}{theorem}
\newtheorem{example}[example]{Example}
\aliascntresetthe{example}
 
\crefname{example}{Example}{Examples}
\crefname{theorem}{Theorem}{Theorems}
\crefname{proposition}{Proposition}{Propositions}
\crefname{corollary}{Corollary}{Corollaries}
\crefname{lemma}{Lemma}{Lemmas}
\crefname{remark}{Remark}{Remarks}

\newcommand{\T}{\mathcal T}
\newcommand{\G}{\mathbb G}
\newcommand{\F}{\mathcal F}
\newcommand{\Hh}{\mathfrak H}
\DeclarePairedDelimiterXPP{\E}[1]{\mathbb E}{\lbrace}{\rbrace}{}{#1}
\DeclarePairedDelimiterXPP{\Ehat}[1]{\widehat{\mathbb E}_{\pi}}{\lbrace}{\rbrace}{}{#1}
\newcommand{\Pp}{\mathbb P}
\newcommand{\dd}{\mathrm d}
\newcommand{\one}{\mathbf 1}
\newcommand{\zero}{\mathbf 0}
\newcommand{\eqd}{\stackrel{\mathrm d}{=}}

\title[Branch-stationary max-stable fields on trees]
{Branch-stationary max-stable fields on rooted trees}
 
\author{Enkelejd Hashorva}
\address{Enkelejd Hashorva, Department of Actuarial Science,
  University of Lausanne,\newline UNIL-Dorigny, 1015 Lausanne, Switzerland}
\email{Enkelejd.Hashorva@unil.ch}
 
\author{Svyatoslav Novikov}
\address{Svyatoslav Novikov, Department of Actuarial Science,
  University of Lausanne,\newline UNIL-Dorigny, 1015 Lausanne, Switzerland}
\email{Svyatoslav.Novikov@unil.ch}
 
\begin{document}
 
\begin{abstract}
We study max-stable random fields on the rooted tree
\(\T=\mathcal A^*\) under shifts to descendant subtrees.
Branch-Brown--Resnick stationarity is characterised through homogeneous
spectral classes, punctured tail measures, and local spectral tail fields.
For lognormal representers, it is equivalent to invariance of the
variogram under addition of a common prefix. We give Gaussian,
max-autoregressive, regenerative cascade, and free-group cluster
constructions, and show that summability on countably branching trees need
not satisfy a zero--one law. We also derive the associated
branch-invariant extreme-value and Archimax copulas.
\end{abstract}
 
 
 
\maketitle

 \section{Introduction}
 
Let \(\T\) be a non-empty countable index set, and let
\(Z=(Z_v)_{v\in\T}\) be a non-negative random field satisfying
\begin{equation}
  \E{Z_v^\alpha}=1,
  \qquad v\in\T
  \label{1}
\end{equation}
for some fixed \(\alpha>0\). Suppose further that
\begin{equation}
  \Pp\left\{\sup_{v\in\T}Z_v>0\right\}=1.
  \label{2}
\end{equation}
The associated max-stable field
\(\eta=(\eta_v)_{v\in\T}\) is specified through its
finite-dimensional distributions by
\begin{equation}
  \Pp\{\eta_v\leq x_v,\ v\in K\}
  =
  \exp\left\{
    -\E*{\max_{v\in K}\frac{Z_v^\alpha}{x_v^\alpha}}
  \right\}
  \label{3}
\end{equation}
for every non-empty finite \(K\subset\T\) and \(x_v>0\), \(v\in K\).
Its margins are unit \(\alpha\)-Fr\'echet; see, e.g.,
\cite{deHaan1984,molchanov2014invariance}.
 
Equivalently, let \((P_i)_{i\geq1}\) be the points of a Poisson
point process on \((0,\infty)\) with intensity
\[
  \alpha p^{-\alpha-1}\dd p
\]
and let \(Z^{(1)},Z^{(2)},\ldots\) be independent copies of \(Z\),
independent of \((P_i)_{i\geq1}\). Then de Haan's spectral
representation yields
\begin{equation}
  (\eta_v)_{v\in\T}
  \eqd
  \left(
    \bigvee_{i\geq1}P_iZ_v^{(i)}
  \right)_{v\in\T}.
  \label{4}
\end{equation}
 
The field \(Z\) is commonly called an \(\alpha\)-spectral process,
or a representer, of \(\eta\). Representers are not unique. Indeed,
if \(S>0\) is independent of \(Z\) and
\(\E{S^\alpha}=1\), then \(SZ\) is another representer of the same
max-stable field.
 
The seminal paper \cite{BrownResnick1977} showed that, for
\(\alpha=1\), the choice
\[
  Z_v
  =
  \exp\left\{
    W_v-\frac{1}{2}\operatorname{Var}(W_v)
  \right\},
  \qquad v\in\mathbb R,
\]
where \(W\) is a two-sided standard Brownian motion, generates a
stationary max-stable process. More generally, if
\(W=(W_v)_{v\in\mathbb R^p}\) is a centred Gaussian process with
stationary increments and \(W_0=0\), then
\[
  Z_v
  =
  \exp\left\{
    W_v-\frac{1}{2}\operatorname{Var}(W_v)
  \right\},
  \qquad v\in\mathbb R^p,
\]
generates a stationary max-stable field; see
\cite{KabluchkoSchlatherdeHaan2009} and
\cite{HashorvaKume2021,Hashorva2024} for related multivariate
extensions.
 
In this paper we consider an analogue on a non-Euclidean index set,
namely a rooted tree. Specifically, unless otherwise stated, we take
\[
  \T=\mathcal A^*
  =
  \{\varnothing\}\cup\bigcup_{n\geq1}\mathcal A^n
\]
the free monoid generated by a non-empty finite or countable alphabet
\(\mathcal A\), and regard the empty word \(\varnothing\) as the root.
A basic instance is the binary tree with
\(\mathcal A=\{L,R\}\), in which case
\[
  \T
  =
  \{\varnothing,L,R,LL,LR,RL,RR,\ldots\}.
\]
 
The monoid operation is concatenation, with \(\varnothing\) as its
identity. The resulting monoid has the natural rooted-tree structure
in which the children of \(u\in\T\) are \(ua\),
\(a\in\mathcal A\). For every \(u\in\T\), the descendant subtree
\(u\T\) is canonically identified with \(\T\) through the map
\(v\mapsto uv\).
 
For a configuration \(f=(f_v)_{v\in\T}\), define its descendant
subtree shift at \(u\in\T\) by
\[
  B^uf=(f_{uv})_{v\in\T}.
\]
A random field \(\eta=(\eta_v)_{v\in\T}\) is called
\emph{branch-stationary} if
\[
  B^u\eta\eqd\eta,
  \qquad u\in\T.
\]
We call a representer \(Z\)
\emph{branch-Brown--Resnick stationary}, abbreviated
\emph{branch-BR stationary}, if the corresponding max-stable field
in \eqref{3} is branch-stationary.
 
Rooted trees are natural index sets in economic models. They arise in
models of organisational and communication hierarchies
\cite{BoltonDewatripont1994,Garicano2000} and can also represent
hierarchically structured choice sets or histories in extensive-form
decision problems. A vertex may represent an agent, an organisational
unit, an alternative, or a history, whereas an edge \(u\to ua\) may
represent a communication link, transition, or action. Every non-root
vertex \(ua\) is canonically identified with its incoming edge
\((u,ua)\). Hence, apart from the root coordinate, a vertex-indexed
field can equivalently be interpreted as a field of shocks attached
to links, transitions, or actions. The descendant subtree below \(u\)
then represents the continuation hierarchy, decision problem, or game
following the history \(u\).
 
The associated max-stable field also has a direct random-utility
interpretation. Define
\[
  \varepsilon_v=\alpha\log\eta_v,
  \qquad v\in\T.
\]
When the coordinates of \(Z\) are strictly positive,
\eqref{4} gives
\begin{equation}
  \varepsilon_v
  =
  \bigvee_{i\geq1}
  \left\{
    \mathcal Q_i+Y_v^{(i)}
  \right\},
  \qquad
  \mathcal Q_i=\alpha\log P_i,
  \quad
  Y_v^{(i)}=\alpha\log Z_v^{(i)}.
  \label{5}
\end{equation}
Thus \(\varepsilon_v\) is the highest score among a Poisson collection
of latent utility profiles. The same Poisson atom may dominate at
several vertices and therefore represents a common extreme shock
propagating through the hierarchy.
 
Since the margins of \(\eta\) are unit \(\alpha\)-Fr\'echet,
each \(\varepsilon_v\) has the standard Gumbel distribution. For every
non-empty finite \(K\subset\T\), define
\[
  \ell_K(x)
  =
  \E*{\max_{v\in K}x_vZ_v^\alpha},
  \qquad
  x=(x_v)_{v\in K}\in[0,\infty)^K.
\]
Then
\[
  \Pp\{\varepsilon_v\leq y_v,\ v\in K\}
  =
  \exp\left\{
    -\ell_K\bigl((e^{-y_v})_{v\in K}\bigr)
  \right\}.
\]
Consequently, for deterministic systematic utilities \(V_v\), the
utilities
\[
  \mathcal U_v=V_v+\varepsilon_v,
  \qquad v\in K,
\]
form a generalised extreme-value random-utility model; see
\cite{McFadden1981,ResnickRoy1991,Dagsvik1994}. When
\(\mathcal A\) is finite, the restriction of \(\varepsilon\) to a
generation \(\mathcal A^n\) provides a finite system of dependent
Gumbel utility shocks for hierarchically related alternatives. For a
countable alphabet, the same interpretation applies to finite choice
subsets of a generation.
 
An illuminating example is provided by a Gaussian branching random
walk. Let \((X_e)\) be independent centred Gaussian random variables
with variance \(\Delta>0\), indexed by the directed edges of the
tree. Write \([\varnothing,v]\) for the collection of edges on the
unique path from the root to \(v\), and define
\[
  W_\varnothing=0,
  \qquad
  W_v=\sum_{e\in[\varnothing,v]}X_e,
  \qquad v\neq\varnothing.
\]
For \(\alpha=1\), put
\[
  Z_v
  =
  \exp\left\{
    W_v-\frac{\Delta|v|}{2}
  \right\},
  \qquad v\in\T.
\]
Then \(\E{Z_v}=1\). Moreover, for every \(u\in\T\),
\[
  Z_{uv}
  =
  Z_u\widetilde Z_v^{(u)},
  \qquad v\in\T,
\]
where \(\widetilde Z^{(u)}\) is a copy of \(Z\) independent of
\(Z_u\). Hence, for every non-empty finite \(K\subset\T\),
\[
  \E*{\max_{v\in K}\frac{Z_{uv}}{x_v}}
  =
  \E{Z_u}\,
  \E*{\max_{v\in K}
       \frac{\widetilde Z_v^{(u)}}{x_v}}
  =
  \E*{\max_{v\in K}\frac{Z_v}{x_v}},
  \qquad
  x_v>0,\quad v\in K.
\]
Thus the corresponding max-stable field is branch-stationary.
 
This leads to our first question.
 
\medskip
 
\noindent
\textbf{Q1.}
\emph{Which lognormal representers generate branch-stationary
max-stable fields on rooted trees, and how can they be characterised?}
 
Of course, if
\[
  B^uZ\eqd Z,
  \qquad u\in\T,
\]
then the corresponding max-stable field is branch-stationary without
any lognormality assumption. In the Euclidean case, however, there are
important non-lognormal representers that are not themselves stationary
but generate stationary max-stable processes on \(\mathbb R^d\) or
\(\mathbb Z^d\); see, e.g.,
\cite{DombryKabluchko2017,DombryHashorvaSoulier2018,
HashorvaKume2021,Hashorva2024}. Such representers are commonly called
Brown--Resnick stationary spectral processes.
 
One direct way to construct Brown--Resnick stationary representers is
through integrable cluster fields and random-shift representations; see,
e.g.,
\cite{DombryKabluchko2017,DombryHashorvaSoulier2018,Hashorva2025}.
These constructions rely on the group structure of the underlying
index set. In the present setting the rooted tree
\(\T=\mathcal A^*\) is only a monoid, and the descendant shifts
\(B^u\) are not invertible on \(\T\).
 
Our second question is therefore:
 
\medskip
 
\noindent
\textbf{Q2.}
\emph{Can we construct broad classes of non-lognormal representers,
including representers that are not themselves branch-stationary, which
generate branch-stationary max-stable fields on rooted trees?}
 
One important characterisation of a Brown--Resnick
\(\alpha\)-spectral process indexed for simplicity by \(\mathbb R\)
is the tilt-shift identity
\begin{equation}
  \E*{
    Z_h^\alpha F(B^hZ)
  }
  =
  \E*{
    Z_0^\alpha F(Z)
  },
  \qquad h\in\mathbb R
  \label{6}
\end{equation}
valid for all non-negative measurable \(0\)-homogeneous functionals
\(F:[0,\infty)^{\mathbb R}\to[0,\infty]\) satisfying
\(F(\zero)=0\), where
\[
  B^hf(t)=f(t+h).
\]
This identity was established in \cite{Htilt}. For lognormal
representers it agrees with the time-change formula for spectral tail
processes in \cite{BojanS}.

 As shown therein, for instance if $Z_v$ is log-normal, the identity \eqref{6}   agrees with the time-change formula of \cite{BojanS} for spectral tail processes. This is not a coincidence, since the corresponding max-stable process $\eta(t),t\in \mathbb R$ is a multivariate regularly varying process, see \cite{BladtHashorvaShevchenko2022,resnick2024art},  and in  view of \cite{BladtHashorvaShevchenko2022} the corresponding homogeneous tail measure can be represented as 
\begin{equation}
  \nu_Z[F]
  :=
  \int F(f)\,\nu_Z(\dd f)
  =
  \int_0^\infty
  \E{F(rZ)}
  \alpha r^{-\alpha-1}\dd r,
  \label{7}
\end{equation}
defined for all non-negative measurable functionals \(F\); see
\cite{BladtHashorvaShevchenko2022,resnick2024art}. By
\cite{DombryHashorvaSoulier2018}, stationarity of the corresponding
max-stable process is equivalent to invariance of \(\nu_Z\) under the
shifts \(B^h\), \(h\in\mathbb R\).
 
Tilting \(Z/Z_h\) by \(Z_h^\alpha\), or equivalently using the tail
measure \(\nu_Z\), gives local spectral tail processes
\(\Theta^{[h]}\), \(h\in\mathbb R\). In the Euclidean group setting,
invertibility of the shifts allows the entire family to be encoded by
a single spectral tail process through the usual time-change formula.
 
The third natural question is therefore:
 
\medskip
 
\noindent
\textbf{Q3.}
\emph{How can branch stationarity be characterised in terms of the tail
measure \(\nu_Z\) and the local spectral tail fields, and under what
conditions does a single spectral tail field determine the whole family?}
 
The extreme-value copulas of the finite-dimensional distributions in
\eqref{3} are uniquely given by
\[
  C_K(u)
  =
  \exp\{-\ell_K(-\log u)\},
  \qquad u\in(0,1]^K.
\]
Branch stationarity is equivalent to common-prefix invariance
\[
  \ell_{aK}(x^a)=\ell_K(x)
\]
for every \(a\in\T\), every non-empty finite \(K\subset\T\), and every
\(x\in[0,\infty)^K\), where
\[
  aK=\{av:v\in K\},
  \qquad
  x^a_{av}=x_v.
\]
It is therefore also equivalent to the corresponding invariance of the
copulas \(C_K\). We represent \(\ell_K\) through the local spectral
tail fields. For every completely monotone Archimedean generator
\(\psi\), this yields the projectively consistent Archimax copulas
\[
  C_K^\psi(u)
  =
  \psi\left(
    \ell_K\bigl(\psi^{-1}(u)\bigr)
  \right),
\]
where \(\psi^{-1}\) acts componentwise. The branch
max-autoregressive construction below has bivariate Marshall--Olkin
copulas, while the free-group cluster construction gives an explicit
class of branch-stationary extreme-value and Archimax copulas.
 
The paper is organised as follows.
\Cref{8} characterises branch-BR stationarity in terms of
homogeneous classes, punctured tail measures, and local spectral tail
fields. Under the root no-hidden-zero condition, the entire family of
local spectral tail fields is recovered from a single root field. The
same section establishes the common-prefix variogram criterion for
lognormal representers and gives Gaussian, max-autoregressive, and
regenerative cascade examples.
\Cref{38} studies forward summability and
constructs gauge-independent branch-invariant tail measures from
integrable clusters on the free-group completion of the tree, including
explicit radial models.
\Cref{56} develops the associated extreme-value and
Archimax copulas, their winner-function representation, and the
positive-stable extension. The proofs are collected in the final
section.

 \section{Characterisation and Examples of Branch-BR stationarity}
\label{8}
 
As in the Introduction, an
{\it \(\alpha\)-spectral representer} \(Z\) is non-negative, satisfies
\(\E{Z_v^\alpha}=1\) for every \(v\in\T\), and obeys
\eqref{2}. Its corresponding max-stable field will
be denoted by \(\eta\).
 
In order to study the branch-BR stationarity of \(Z\), we introduce
some additional notation. Put \(\F=[0,\infty)^\T\) and equip it with
the product \(\sigma\)-field \(\mathcal B(\F)\).
 
Hereafter \(\Hh\) is the class of non-negative measurable functionals
\(F:\F\to[0,\infty]\), and \(\Hh_\beta\), \(\beta\in[0,\infty)\), is
the subclass of \(F\in\Hh\) such that \(F(\zero)=0\) and \(F\) is
\(\beta\)-homogeneous, i.e.,
\[
  F(cf)=c^\beta F(f),
  \qquad f\in\F,\quad c>0,
\]
where \(\zero=(0)_{v\in\T}\) denotes the identically zero field, the
zero element of \(\F\).
 
A distinction from the usual group setting is that a branch shift may
send a non-zero field to the zero field, and therefore \(B^aZ\) need
not be non-zero almost surely. Indeed, \(B^a\) retains only the
coordinates in the descendant subtree \(a\T\), so it is possible that
\(Z\neq\zero\) while \(B^aZ=\zero\). Consequently,
\eqref{2} does not imply
\(\Pp\{B^aZ\neq\zero\}=1\).
 
For two non-negative fields \(X,Y\), motivated by
\cite{molchanov2014invariance,HashorvaKume2021}, write
\(X\sim_\alpha Y\) when
\[
  \E{F(X)}=\E{F(Y)},
  \qquad F\in\Hh_\alpha.
\]
Let
\[
  \F^\circ=\F\setminus\{\zero\}
\]
equipped with the trace \(\sigma\)-field
\(\mathcal B(\F^\circ)\). For a non-negative random field
\(X\) on \(\T\), define its radial tail measure on \(\F^\circ\) by
\begin{equation}
  \nu_X^\circ(A)
  =
  \int_0^\infty
  \Pp\{rX\in A\}
  \alpha r^{-\alpha-1}\dd r,
  \qquad
  A\in\mathcal B(\F^\circ).
  \label{9}
\end{equation}
The puncturing is essential only when \(X\) may equal the zero field
with positive probability. Indeed, if one applies the radial formula
on all of \(\F\), then
\[
  \int_0^\infty
  \Pp\{rX=\zero\}
  \alpha r^{-\alpha-1}\dd r
  =
  \begin{cases}
    0, & \Pp\{X=\zero\}=0,\\
    \infty, & \Pp\{X=\zero\}>0.
  \end{cases}
\]
Hence, if \(\Pp\{X=\zero\}>0\), the unpunctured radial measure has an
infinite atom at \(\zero\) and is not \(\sigma\)-finite on \(\F\).
For the standing representer \(Z\), condition
\eqref{2} implies that
\(\nu_Z^\circ\) is simply the restriction of \(\nu_Z\) to
\(\F^\circ\). By contrast, \(B^aZ\) may equal \(\zero\) with positive
probability and hence its relevant tail measure is
\(\nu_{B^aZ}^\circ\). This punctured measure is nevertheless
\(\sigma\)-finite, since
\[
  \nu_{B^aZ}^\circ\{f:f_v>1/m\}
  =
  m^\alpha\E{Z_{av}^\alpha}
  =
  m^\alpha,
  \qquad
  v\in\T,\quad m\geq1
\]
and these sets form a countable cover of \(\F^\circ\).
 
\begin{lemma}
\label{10}
For any two non-negative random fields \(X\) and \(Y\) on \(\T\)
\[
  X\sim_\alpha Y
  \quad\Longleftrightarrow\quad
  \nu_X^\circ=\nu_Y^\circ.
\]
\end{lemma}
 
The tail measure of the standing \(\alpha\)-spectral representer \(Z\),
defined through \eqref{7}, satisfies
\[
  \nu_Z\{\zero\}=0,
  \qquad
  \nu_Z\{f:f_v>1\}
  =
  \E{Z_v^\alpha}
  =
  1,
  \qquad v\in\T.
\]
 
 
For \(u\in\T\), define the local spectral tail field
\(\Theta^{[u]}\) at \(u\) by
\begin{equation}
  \E{H(\Theta^{[u]})}
  =
  \E*{
    \one_{\{Z_u>0\}}
    Z_u^\alpha
    H\left(\frac{Z}{Z_u}\right)
  }, \qquad\forall H\in\Hh.
  \label{11}
\end{equation}  The integrand in
\eqref{11} is defined to be zero on
\(\{Z_u=0\}\). Throughout, \(0\cdot\infty\) in such non-negative
tilted expressions is interpreted as zero. Since
\(\E{Z_u^\alpha}=1\), this defines a probability law, and
\[
  \Theta_u^{[u]}=1
  \qquad\text{almost surely}.
\]
 
The relation between two local spectral tail fields does not require
branch stationarity. In the present normalization, the general
change-of-root formula reads
\begin{equation}
  \E*{
    \one_{\{\Theta_t^{[h]}>0\}}
    (\Theta_t^{[h]})^\alpha
    H\left(
      \frac{\Theta^{[h]}}{\Theta_t^{[h]}}
    \right)
  }
  =
  \E*{
    \one_{\{\Theta_h^{[t]}>0\}}
    H(\Theta^{[t]})
  },
  \qquad
  h,t\in\T,\quad H\in\Hh.
  \label{12}
\end{equation}
Each ratio is evaluated only on the event indicated by the
corresponding indicator, and the integrand is defined to be zero
outside that event. This is the coordinate version of the
change-of-root identity in
\cite[Proposition~3.6]{BladtHashorvaShevchenko2022}; see also
\cite{MR4153591}.
 
\begin{proposition}
\label{13}
Let \(Z\) be an \(\alpha\)-spectral representer satisfying
\eqref{2}, and let \(\eta\) be given by
\eqref{3}. The following statements are
equivalent:
\begin{enumerate}[(i)]
\item \(\eta\) is branch-stationary;
\item \(B^aZ\sim_\alpha Z\) for every \(a\in\T\);
\item the punctured tail measure is branch-invariant, i.e.,
\begin{equation}
  \nu_Z^\circ\bigl((B^a)^{-1}A\bigr)
  =
  \nu_Z^\circ(A),
  \qquad
  A\in\mathcal B(\F^\circ),\quad a\in\T,
  \label{14}
\end{equation}
where
\[
  (B^a)^{-1}A
  =
  \{f\in\F^\circ:B^af\in A\};
\]
\item the local spectral tail fields satisfy
\begin{equation}
  B^a\Theta^{[au]}\eqd\Theta^{[u]},
  \qquad a,u\in\T;
  \label{15}
\end{equation}
\item for every \(a,u\in\T\) and every \(F\in\Hh_0\),
\begin{equation}
  \E*{
    Z_{au}^{\alpha}F(B^aZ)
  }
  =
  \E*{
    Z_u^{\alpha}F(Z)
  }.
  \label{16}
\end{equation}
\end{enumerate}
\end{proposition}
 
It is useful to distinguish the standing non-zero condition
\eqref{2} from the root no-hidden-zero condition
\begin{equation}
  \Pp\{Z_\varnothing=0,\ Z_v>0\}=0,
  \qquad v\in\T.
  \label{17}
\end{equation}
Under \eqref{2}, the countability of \(\T\)
makes \eqref{17} equivalent to
\[
  \Pp\{Z_\varnothing>0\}=1.
\]
Indeed, on \(\{Z_\varnothing=0\}\), condition
\eqref{2} implies that \(Z_v>0\) for at least one
\(v\in\T\) almost surely. In particular, \eqref{17} is
automatically satisfied by lognormal spectral processes.
 
Assume that \(Z\) satisfies \eqref{17} and put
\[
  \Theta=\Theta^{[\varnothing]}.
\]
For every \(F\in\Hh_\alpha\), the definition of the root tilt and
\(\alpha\)-homogeneity imply
\[
  \E{F(\Theta)}
  =
  \E*{
    Z_\varnothing^\alpha
    F\left(\frac{Z}{Z_\varnothing}\right)
  }
  =
  \E{F(Z)}
\] and thus 
\begin{equation}
  \Theta\sim_\alpha Z,
  \qquad
  \nu_\Theta^\circ=\nu_Z^\circ.
  \label{18}
\end{equation}
In particular, \(\Theta\) is an \(\alpha\)-spectral representer of the
same max-stable field as \(Z\). Furthermore, for every \(t\in\T\)
\[
\begin{aligned}
  \Pp\{\Theta_\varnothing^{[t]}=0\}
  &=
  \E*{
    \one_{\{Z_t>0\}}
    Z_t^\alpha
    \one_{\{Z_\varnothing/Z_t=0\}}
  } =
  \E*{
    Z_t^\alpha\one_{\{Z_\varnothing=0\}}
  }
  =
  0.
\end{aligned}
\]
Hence the change-of-root formula with \(h=\varnothing\) reduces to
\begin{equation}
  \E{H(\Theta^{[t]})}
  =
  \E*{
    \one_{\{\Theta_t>0\}}
    \Theta_t^\alpha
    H\left(\frac{\Theta}{\Theta_t}\right)
  },
  \qquad
  t\in\T,\quad H\in\Hh.
  \label{19}
\end{equation}
 
\begin{corollary} 
\label{20}
Let \(Z\) be an \(\alpha\)-spectral representer satisfying
\eqref{2} and \eqref{17}, and put
\(\Theta=\Theta^{[\varnothing]}\). Then \(Z\) is branch-BR stationary
if and only if
\begin{equation}
  \E*{
    \one_{\{\Theta_a>0\}}
    \Theta_a^\alpha
    H\left(\frac{B^a\Theta}{\Theta_a}\right)
  }
  =
  \E{H(\Theta)},
  \qquad
  a\in\T,\quad H\in\Hh.
  \label{21}
\end{equation}
Under these equivalent conditions
\begin{equation}
  \Pp\{\Theta_u=0,\ \Theta_{uv}>0\}
  =
  \Pp\{Z_u=0,\ Z_{uv}>0\}
  =
  0,
  \qquad u,v\in\T.
  \label{22}
\end{equation}
\end{corollary}
 
\subsection{The lognormal variogram criterion}
 
For notation simplicity, we retain the unit-Fr\'echet normalization
used in the Introduction, i.e., \(\alpha=1\) in this subsection. Let
\(W=(W_v)_{v\in\T}\) be a centered Gaussian field, put
\[
  \sigma_v^2=\operatorname{Var}(W_v),
  \qquad
  \Gamma(v,w)=\E{(W_v-W_w)^2},
\]
and define
\begin{equation}
  Z_v
  =
  \exp\left\{
    W_v-\sigma_v^2/2
  \right\}.
  \label{23}
\end{equation}
By construction, \(Z_v>0\) almost surely and \(\E{Z_v}=1\) for every
\(v\in\T\).
 
Replacing \(W_v\) by \(W_v-W_\varnothing\) leaves \(\Gamma\)
unchanged. Since the tail measure generated by a normalized lognormal
representer is determined by its variogram, as follows from the
Gaussian calculation below, the corresponding anchored representer
generates the same max-stable field. We may therefore use the anchored
version \(W_\varnothing=0\) without loss of generality. Then
\(Z_\varnothing=1\), the root spectral tail field \(\Theta\) has the
same law as \(Z\), and the branch-BR criterion can be expressed
entirely in terms of \(\Theta\).
 
More precisely, for \(u\in\T\), put
\[
  X_v^{(u)}=W_{uv}-W_u,
  \qquad v\in\T.
\]
Gaussian exponential tilting, see e.g., \cite{Htilt} yields
\begin{equation}
  B^u\Theta^{[u]}
  \eqd
  \left(
    \exp\left\{
      X_v^{(u)}
      -\frac12\operatorname{Var}(X_v^{(u)})
    \right\}
  \right)_{v\in\T}.
  \label{24}
\end{equation}
 
\begin{lemma}
\label{25}
For the process in \eqref{23}, the following
statements are equivalent:
\begin{enumerate}[(i)]
\item \(Z\) is branch-BR stationary;
\item \(W\) has branch-stationary increments, i.e.,
\[
  (W_{uv}-W_u)_{v\in\T}
  \eqd
  (W_v-W_\varnothing)_{v\in\T},
  \qquad u\in\T;
\]
\item the variogram is invariant under the addition of a common prefix,
i.e.,
\begin{equation}
  \Gamma(uv,uw)=\Gamma(v,w),
  \qquad u,v,w\in\T.
  \label{26}
\end{equation}
\end{enumerate}
\end{lemma}
 
\begin{remark}[Vector-valued lognormal fields]
Let
\[
  W_v=(W_{v,1},\ldots,W_{v,d}),
  \qquad d\in\mathbb N,\quad v\in\T,
\]
be a centered Gaussian field and set
\[
  Z_{v,i}
  =
  \exp\left\{
    W_{v,i}
    -\frac12\operatorname{Var}(W_{v,i})
  \right\},
  \qquad 1\leq i\leq d.
\]
Using common Poisson points with componentwise maxima gives the
associated \(d\)-variate max-stable field. Define the
pseudo-variograms of \(W\) by
\[
  \gamma_{ij}(v,w)
  =
  \E{(W_{v,i}-W_{w,j})^2},
  \qquad 1\leq i,j\leq d.
\]
The associated field is branch-stationary if and only if
\begin{equation}
  \gamma_{ij}(uv,uw)=\gamma_{ij}(v,w),
  \qquad
  u,v,w\in\T,\quad 1\leq i,j\leq d.
  \label{27}
\end{equation} 
\end{remark}
 
\begin{remark}
Let \(q=(q_a)_{a\in\mathcal A}\) satisfy \(q_a>0\). For
\(v,w\in\T\), let \(r=v\wedge w\) be their longest common prefix and
write
\[
  v=ra_1\cdots a_m,
  \qquad
  w=rb_1\cdots b_n.
\]
The letter-weighted genealogical metric
\begin{equation}
  d_q(v,w)
  =
  \sum_{j=1}^m q_{a_j}
  +
  \sum_{j=1}^n q_{b_j}
  \label{28}
\end{equation}
is invariant under common prefixes. It is also conditionally negative
definite. Indeed, if
\((e_{x,a})_{x\in\T,a\in\mathcal A}\) is the canonical orthonormal
basis indexed by the edges and \(v=a_1\cdots a_k\), then
\[
  \iota_q(v)
  =
  \sum_{j=1}^k
  \sqrt{q_{a_j}}\,
  e_{a_1\cdots a_{j-1},a_j}
\]
satisfies
\[
  d_q(v,w)
  =
  \|\iota_q(v)-\iota_q(w)\|_2^2.
\]
For every Bernstein function \(\phi\) with \(\phi(0)=0\), the function
\begin{equation}
  \Gamma(v,w)=\phi\bigl(d_q(v,w)\bigr)
  \label{29}
\end{equation}
is an admissible branch-invariant variogram; see, for example,
\cite{SchillingSongVondracek2012}. Concrete choices include
\[
  \phi(t)=ct^\beta,\quad 0<\beta\leq1,
  \qquad
  \phi(t)=c(1-e^{-\lambda t}),
  \qquad
  \phi(t)=c\log(1+\lambda t),
\]
where \(c,\lambda>0\). The choice \(\phi(t)=t\) gives a Gaussian
branching random walk in which an edge leading to a child labelled
\(a\) has variance \(q_a\). Thus every variogram in
\eqref{29} yields a branch-BR lognormal
spectral field by \Cref{25}. This extends the
usual Brown--Resnick construction from Gaussian fields with stationary
increments
\cite{BrownResnick1977,KabluchkoSchlatherdeHaan2009}.
\end{remark}
 
\begin{example}
Assume that \(q_a=1\) for every \(a\in\mathcal A\), and write
\(d_\T=d_q\) for the graph distance. For \(0<H\leq1/2\), let \(W^H\)
be the centered Gaussian field with \(W^H_\varnothing=0\) and
covariance
\[
  \E{W^H_vW^H_w}
  =
  \frac12
  \left\{
    |v|^{2H}
    +
    |w|^{2H}
    -
    d_\T(v,w)^{2H}
  \right\},
  \qquad v,w\in\T.
\]
It exists because \(t\mapsto t^{2H}\) is a Bernstein function. Its
restriction to every infinite ray is a discrete fractional Brownian
motion with Hurst parameter \(H\). Since its variogram
\(d_\T(v,w)^{2H}\) is invariant under common prefixes,
\[
  Z^H_v
  =
  \exp\left\{
    W^H_v-\frac12|v|^{2H}
  \right\}
\]
is branch-BR stationary. The case \(H=1/2\) is the Gaussian branching
random walk with unit edge variances.
\end{example}
 
Branch stationarity is not specific to lognormal spectral models. The
following recursion is the rooted-tree analogue of the classical
max-autoregressive construction of \cite{DavisResnick1989}.
 
\begin{example} 
\label{30}
Let \((\xi_v)_{v\in\T}\) be independent unit
\(\alpha\)-Fr\'echet random variables. For \(a\in\mathcal A\), fix
\(c_a\in[0,1]\), set
\[
  d_a=(1-c_a^\alpha)^{1/\alpha},
\]
and define
\begin{equation}
  \eta_\varnothing=\xi_\varnothing,
  \qquad
  \eta_{ua}
  =
  \max\{c_a\eta_u,d_a\xi_{ua}\},
  \qquad
  u\in\T,\quad a\in\mathcal A.
  \label{31}
\end{equation}
Then \(\eta\) is a branch-stationary max-stable field with unit
\(\alpha\)-Fr\'echet margins.
 
Indeed, \(\eta_u\) is independent of \(\xi_{ua}\), and
\(c_a^\alpha+d_a^\alpha=1\). Induction over the generations therefore
gives unit \(\alpha\)-Fr\'echet margins. Moreover, every
finite-dimensional vector of \(\eta\) is a finite max-linear
transform of independent \(\alpha\)-Fr\'echet random variables and is
hence max-stable.
 
To prove branch stationarity, fix \(u\in\T\) and define
\[
  \widetilde\xi_\varnothing^{(u)}=\eta_u,
  \qquad
  \widetilde\xi_v^{(u)}=\xi_{uv},
  \quad v\neq\varnothing.
\]
Since \(\eta_u\) depends only on innovations along its ancestral path,
it is independent of all strict-descendant innovations. Consequently,
\[
  (\widetilde\xi_v^{(u)})_{v\in\T}
  \eqd
  (\xi_v)_{v\in\T}.
\]
The field \(B^u\eta\) is generated from
\((\widetilde\xi_v^{(u)})_{v\in\T}\) by the same recursion
\eqref{31}, hence
\[
  B^u\eta\eqd\eta,
  \qquad u\in\T.
\]
 
For later use, if \(p=a_1\cdots a_m\), set
\[
  c(p)=\prod_{j=1}^m c_{a_j},
  \qquad
  c(\varnothing)=1.
\]
Iterating \eqref{31} gives, for every \(r\in\T\),
\begin{equation}
  \eta_{rp}
  =
  \max\left\{
    c(p)\eta_r,\,
    \bigvee_{j=1}^m
    c(a_{j+1}\cdots a_m)d_{a_j}
    \xi_{ra_1\cdots a_j}
  \right\},
  \label{32}
\end{equation}
where the maximum over the empty set is zero.
\end{example}
 
\begin{remark} 
\label{33}
Let \(v,w\in\T\) be distinct, let \(r=v\wedge w\), and write
\[
  v=rp,
  \qquad
  w=rq.
\]
Set
\[
  A=c(p)^\alpha,
  \qquad
  B=c(q)^\alpha.
\]
By \eqref{32}
\[
  \eta_v=\max\{c(p)\eta_r,R_{r,p}\},
  \qquad
  \eta_w=\max\{c(q)\eta_r,R_{r,q}\},
\]
where \(\eta_r,R_{r,p},R_{r,q}\) are independent and
\[
  \Pp\{R_{r,p}\leq x\}
  =
  \exp\{-(1-A)x^{-\alpha}\},
  \qquad
  \Pp\{R_{r,q}\leq x\}
  =
  \exp\{-(1-B)x^{-\alpha}\}.
\]
Here the exponent coefficients follow by telescoping in
\eqref{32}; if \(p=\varnothing\) or
\(q=\varnothing\), the corresponding residual is zero.
 
Consequently, the stable tail dependence function of
\((\eta_v,\eta_w)\) is
\begin{equation}
  \ell_{\{v,w\}}(z_1,z_2)
  =
  (1-A)z_1
  +
  (1-B)z_2
  +
  \max\{Az_1,Bz_2\}.
  \label{34}
\end{equation}
Hence its extreme-value copula is the Marshall--Olkin copula
\begin{equation}
  C_{\{v,w\}}(s,t)
  =
  s^{1-A}t^{1-B}\min\{s^A,t^B\},
  \qquad s,t\in(0,1],
  \label{35}
\end{equation}
and
\[
  \lambda_U(v,w)
  =
  2-\ell_{\{v,w\}}(1,1)
  =
  \min\{A,B\}.
\]
\end{remark}
 
\subsection{Branch-regenerative cascades}
\label{36}
 
The calculation in the introductory example rests on the following
regenerative factorisation. Suppose that \(Z\) is a \(1\)-spectral
process with \(Z_\varnothing=1\) and, for every \(u\in\T\),
\begin{equation}
  B^uZ\eqd Z_uZ^{(u)},
  \label{37}
\end{equation}
where \(Z^{(u)}\) has the same law as \(Z\) and is independent of
\(Z_u\).
 
For every \(F\in\Hh_1\), homogeneity and \(F(\zero)=0\) imply
\[
  F(Z_uZ^{(u)})
  =
  \begin{cases}
    Z_uF(Z^{(u)}), & Z_u>0,\\
    0, & Z_u=0.
  \end{cases}
\]
Consequently, independence and Tonelli's theorem yield
\[
\begin{aligned}
  \E{F(B^uZ)}
  &=
  \E{F(Z_uZ^{(u)})} =
  \E*{
    \one_{\{Z_u>0\}}
    Z_uF(Z^{(u)})
  } =
  \E{Z_u}\,\E{F(Z)}
  =
  \E{F(Z)}.
\end{aligned}
\]
Hence \(Z\) is branch-BR stationary.
 
Multiplicative cascades form a broad class satisfying
\eqref{37}. Let
$
  (M_{u,a})_{a\in\mathcal A},
   u\in\T$ 
be identically distributed non-negative offspring vectors that are
independent over \(u\); the coordinates within each offspring vector
may be dependent. Assume that
\[
  \E{M_{u,a}}=1,
  \qquad a\in\mathcal A.
\]
Starting with \(Z_\varnothing=1\), define recursively
\[
  Z_{ua}=Z_uM_{u,a},
  \qquad
  u\in\T,\quad a\in\mathcal A.
\]
Along every ancestral path, the multipliers are drawn from distinct
offspring vectors and are therefore independent. Hence
\(\E{Z_v}=1\) for every \(v\in\T\), so \(Z\) is a
\(1\)-spectral process. Moreover, the offspring vectors indexed by
\(u\) and its descendants are independent of \(Z_u\) and form an
independent copy of the original cascade. Therefore
\eqref{37} holds. Zeros are
hereditary in this construction, while \eqref{17} is automatic.
 
The logarithmic formulation includes the usual branching random walks
and cascades with infinitely divisible log-increments. Suppose that
the offspring log-increment vectors
\[
  (X_{u,a})_{a\in\mathcal A},
  \qquad u\in\T
\]
are identically distributed and independent over \(u\), and that
\[
  \kappa_a
  =
  \log\E{e^{X_{u,a}}}
  <
  \infty,
  \qquad a\in\mathcal A.
\]
Then
\[
  M_{u,a}
  =
  \exp\{X_{u,a}-\kappa_a\}
\]
defines a branch-BR cascade. Gaussian log-increments give a lognormal
cascade, while infinitely divisible offspring vectors give the
corresponding L\'evy-type constructions.

\section{Cluster constructions}
\label{38}
 
For \(f\in\F\), define its forward \(\alpha\)-mass and exceedance count by
\[
  S^+(f)=\sum_{v\in\T}f_v^\alpha,
  \qquad
  N^+(f)=\sum_{v\in\T}\one_{\{f_v>1\}}.
\]
On an extension carrying \(\Theta\), let \(R_{\mathrm P}\) be independent of
\(\Theta\) with the standard \(\alpha\)-Pareto law
\[
  \Pp\{R_{\mathrm P}>r\}=r^{-\alpha},
  \qquad r\geq1.
\]
 
\begin{lemma}
\label{39}
Let \(Z\) be an \(\alpha\)-spectral process satisfying
\eqref{2}, let
\(\Theta=\Theta^{[\varnothing]}\), and put
\(Y=R_{\mathrm P}\Theta\).
If \(Z\) satisfies \eqref{17}, the following statements are
equivalent:
\begin{enumerate}[(i)]
\item \(\Pp\{S^+(Z)<\infty\}=1\);
\item \(\Pp\{S^+(\Theta)<\infty\}=1\);
\item \(\Pp\{S^+(Y)<\infty\}=1\);
\item \(\Pp\{\E{N^+(Y)\mid\Theta}<\infty\}=1\).
\end{enumerate}
The same equivalences hold with \(<\infty\) replaced throughout by
\(=\infty\).
\end{lemma}
 
The root-tilt equivalence does not impose a zero--one law when the tree is
countably branching.
 
\begin{lemma}
\label{40}
Assume that \(\alpha=1\). For every \(p\in(0,1)\), there is a
lognormal branch-BR representer \(Z\) on a countably branching tree such
that
\[
  \Pp\{S^+(Z)<\infty\}=p.
\]
\end{lemma}
 
\begin{remark}
\label{41}
The preceding construction does not extend directly to a stationary
Brown--Resnick process indexed by \(\mathbb Z\); see also
\cite[Question~6.2]{WangStoev2010}.
\end{remark}
 
We now construct branch-invariant tail measures supported on forward
\(\alpha\)-summable fields.
 
Let
\[
  \G=\mathbb F_{\mathcal A}
\]
be the free group generated by \(\mathcal A\). For general
\(\mathcal A\), the free monoid
\[
  \T=\mathcal A^*
\]
is naturally embedded in \(\G=\mathbb F_{\mathcal A}\) as the positive
submonoid consisting of all reduced words containing no inverse letters.
For every \(u\in\T\), right multiplication
\[
  R_u:\G\to\G,
  \qquad
  R_u(g)=gu
\]
is a bijection, with inverse
\[
  R_u^{-1}(h)=hu^{-1}.
\]
By contrast, the corresponding map
\[
  \T\to\T,
  \qquad
  v\mapsto vu
\]
is injective but, unless \(u=\varnothing\), not surjective. Thus
\(\G\) consists of reduced finite words in the letters
\(a\in\mathcal A\) and their formal inverses.
 
If \(\mathcal A=\{L,R\}\), then
\[
  \mathcal A^{-1}=\{L^{-1},R^{-1}\},
\]
and
\[
  \G=\mathbb F_{\{L,R\}}
  =
  \{\varnothing\}
  \cup
  \left\{
    x_1\cdots x_n:
    n\geq1,\quad
    x_i\in\{L,R,L^{-1},R^{-1}\},\quad
    x_{i+1}\neq x_i^{-1}
    \text{ for }1\leq i<n
  \right\}.
\]
Words such as \(RLLLL\) and \(RL^{-1}RLL\) are reduced and hence belong
to \(\G\). By contrast, \(RL^{-1}LR\) is not reduced, since the adjacent
letters \(L^{-1}L\) cancel; it represents the same group element as
\(RR\). Multiplication in \(\G\) is concatenation followed by the
cancellation of adjacent inverse pairs
\[
  LL^{-1}=L^{-1}L=RR^{-1}=R^{-1}R=\varnothing.
\]
The rooted binary tree
\[
  \T=\{L,R\}^*
\]
is naturally embedded in \(\G\) as the positive submonoid consisting
of words containing only \(L\) and \(R\).
 
Let \(Q=(Q_g)_{g\in\G}\) be a non-negative random field satisfying
\begin{equation}
  0<c_Q
  :=
  \E*{\sum_{g\in\G}Q_g^\alpha}
  <\infty.
  \label{42}
\end{equation}
We refer to \(Q\) as an integrable cluster field. Condition
\eqref{42} also implies
\[
  \sum_{g\in\G}Q_g^\alpha<\infty
  \qquad\text{almost surely}.
\]
 
For \(g\in\G\), let
\[
  Q^{[g]}=(Q_{gv})_{v\in\T}
\]
denote the positive-tree restriction of the cluster viewed from the anchor
\(g\). Then, for every \(u\in\T\),
\[
  B^uQ^{[g]}
  =
  \bigl(Q_{guv}\bigr)_{v\in\T}
  =
  Q^{[gu]}.
\]
Consequently, shifting to the descendant subtree below \(u\)
corresponds to replacing the cluster anchor \(g\) by \(gu\). Since
\(g\mapsto gu\) is a bijection of \(\G\), sums over all cluster
anchors are unchanged under descendant shifts. This is the mechanism
that yields branch invariance without a boundary term.
 
Fix a strictly positive probability mass function
\(\pi=(\pi_v)_{v\in\T}\) and define
\[
  \rho_\pi:\F\to[0,\infty],
  \qquad
  \rho_\pi(f)
  =
  \left(
    \sum_{v\in\T}\pi_vf_v^\alpha
  \right)^{1/\alpha}.
\]
By \eqref{42},
\[
  \rho_\pi(Q^{[g]})^\alpha
  \leq
  \sum_{h\in\G}Q_h^\alpha
  <\infty
  \qquad\text{almost surely}
\]
for every \(g\in\G\). The next result gives the cluster construction.
 
\begin{proposition}
\label{43}
Assume \eqref{42} and define the law of a random field
\(Z^{Q,\pi}\) on \(\T\) by
\begin{equation}
  \Ehat{H(Z^{Q,\pi})}
  =
  \frac1{c_Q}
  \sum_{g\in\G}
  \E*{
    \rho_\pi(Q^{[g]})^\alpha
    H\left(
      \frac{Q^{[g]}}{\rho_\pi(Q^{[g]})}
    \right)
  }, \qquad H \in \mathfrak H.
  \label{44}
\end{equation}  On
\(\{\rho_\pi(Q^{[g]})=0\}\), the ratio in
\eqref{44} may be defined arbitrarily and the
corresponding summand is understood to be zero. If
\(\widehat{\Pp}_\pi\) stands for the resulting probability measure,
then:
\begin{enumerate}[(i)]
\item \eqref{44} is a probability law,
\[
  \rho_\pi(Z^{Q,\pi})=1
  \quad\text{almost surely},
  \qquad
  \Ehat{(Z_v^{Q,\pi})^\alpha}=1,
  \quad v\in\T;
\]
in particular, \(Z^{Q,\pi}\neq\zero\) almost surely;
\item \(Z^{Q,\pi}\) is almost surely \(\alpha\)-summable, i.e.,
\[
  \sum_{v\in\T}(Z_v^{Q,\pi})^\alpha<\infty;
\]
\item for every \(F\in\Hh_\alpha\),
\begin{equation}
  \Ehat{F(Z^{Q,\pi})}
  =
  \frac1{c_Q}
  \sum_{g\in\G}\E{F(Q^{[g]})}.
  \label{45}
\end{equation}
Consequently, \(Z^{Q,\pi}\) is branch-BR stationary, and its
homogeneous class does not depend on \(\pi\);
\item its tail measure is
\begin{equation}
  \nu_Q(A)
  =
  \frac1{c_Q}
  \sum_{g\in\G}
  \int_0^\infty
  \Pp\{rQ^{[g]}\in A,\ Q^{[g]}\neq\zero\}
  \alpha r^{-\alpha-1}\dd r,
  \qquad A\in\mathcal B(\F),
  \label{46}
\end{equation}
and satisfies
\begin{equation}
  \nu_Q\bigl((B^u)^{-1}A\bigr)
  =
  \nu_Q(A),
  \qquad
  A\in\mathcal B(\F^\circ),\quad u\in\T;
  \label{47}
\end{equation}
\item the root spectral tail field \(\Theta^Q\) is independent of
\(\pi\) and satisfies
\begin{equation}
  \E{H(\Theta^Q)}
  =
  \frac1{c_Q}
  \sum_{g\in\G}
  \E*{
    \one_{\{Q_g>0\}}
    Q_g^\alpha
    H\left(
      \left(\frac{Q_{gv}}{Q_g}\right)_{v\in\T}
    \right)
  }, \qquad F\in\Hh.
  \label{48}
\end{equation}
The summand is understood to be zero on \(\{Q_g=0\}\). In particular, $
  \Pp\{\sum_{v\in\T}(\Theta_v^Q)^\alpha<\infty\}=1$.
\end{enumerate}
\end{proposition}
 
\begin{corollary}
\label{49}
For the representer \(Z^{Q,\pi}\),
\[
  Z_\varnothing^{Q,\pi}>0
  \qquad\text{almost surely}
\]
if and only if
\begin{equation}
  \Pp\{Q_g=0,\ Q_{gv}>0\}=0,
  \qquad g\in\G,\quad v\in\T.
  \label{50}
\end{equation}
\end{corollary}
 
\begin{remark}
\label{51}
Let
\[
  \sum_{i\geq1}\delta_{(R_i,G_i,Q^{(i)})}
\]
be a Poisson point process on
\((0,\infty)\times\G\times[0,\infty)^\G\) with intensity
\[
  c_Q^{-1}\alpha r^{-\alpha-1}\dd r
  \otimes\#(\dd g)\otimes\Pp_Q(\dd q),
\]
where \(\#\) is counting measure on \(\G\) and \(\Pp_Q\) is the law
of \(Q\). Then
\begin{equation}
  \eta_v^Q
  =
  \bigvee_{i\geq1}R_iQ^{(i)}_{G_iv},
  \qquad v\in\T
  \label{52}
\end{equation}
is a branch-stationary max-stable field with unit
\(\alpha\)-Fr\'echet margins and exponent measure \(\nu_Q\).
Atoms for which \(Q^{(i),[G_i]}=\zero\) may be discarded, since they
do not contribute to the maximum; the resulting exponent measure on
\(\F^\circ\) is \(\nu_Q\).
\end{remark}
 
\begin{example}
\label{53}
Let \(\mathcal A=\{L,R\}\) and
\[
  \G=\mathbb F_{\{L,R\}}.
\]
Write
\[
  d(c,c')=|c^{-1}c'|,
  \qquad c,c'\in\G,
\]
for its word metric. The positive words form the observed binary tree
\(\T=\{L,R\}^*\). For \(0<\rho<1/3\), let
\[
  \varphi_\rho(n)=\rho^{n/\alpha},
  \qquad n\in\mathbb N_0.
\]
Since the sphere of radius \(n\geq1\) in \(\G\) has
\(4\cdot3^{n-1}\) elements,
\begin{equation}
  c_\rho
  :=
  \sum_{c\in\G}
  \varphi_\rho(d(c,\varnothing))^\alpha
  =
  1+4\sum_{n\geq1}3^{n-1}\rho^n
  =
  \frac{1+\rho}{1-3\rho}.
  \label{54}
\end{equation}
Take the deterministic cluster
\[
  Q_g=\varphi_\rho(|g|).
\]
After the reindexing \(g=c^{-1}\),
\Cref{51} gives a max-stable field \(\zeta\)
with exponent functions
\begin{equation}
\begin{split}
  -\log\Pp\{\zeta_v\leq x_v,\ v\in K\}
  &=
  V_K^\rho(x) \\
  &=
  \frac1{c_\rho}
  \sum_{c\in\G}
  \max_{v\in K}
  \left(
    \frac{\varphi_\rho(d(c,v))}{x_v}
  \right)^\alpha.
\end{split}
  \label{55}
\end{equation}
The normalization in \eqref{54} gives unit
\(\alpha\)-Fr\'echet margins. For \(u\in\T\), left translation
\(c\mapsto u^{-1}c\) is an isometric bijection and
\[
  d(c,uv)=d(u^{-1}c,v).
\]
Reindexing the sum in \eqref{55} yields
\[
  -\log\Pp\{\zeta_{uv}\leq x_v,\ v\in K\}
  =
  V_K^\rho(x)
\]
and therefore proves branch stationarity.
 
The same construction applies to every non-zero radial profile
\(\varphi:\mathbb N_0\to[0,\infty)\) satisfying
\[
  \sum_{c\in\G}\varphi(|c|)^\alpha<\infty.
\]
Two useful finite-range choices are
\[
  \varphi_R(n)=\one_{\{n\leq R\}},
  \qquad
  \varphi_{\beta,R}(n)
  =
  \exp\left\{-\frac{\beta n^2}{2\alpha}\right\}
  \one_{\{n\leq R\}}.
\]
The case \(R=0\) gives independent Fr\'echet coordinates, while finite
\(R\) gives finite-range dependence. Because these profiles have
zeros, the corresponding gauge-normalized representers need not
satisfy \eqref{17}, although branch stationarity is unaffected.
For a regular \(b\)-ary tree, the same construction uses sphere sizes
\[
  2b(2b-1)^{n-1},
  \qquad n\geq1.
\]
Thus the geometric profile requires
\[
  0<\rho<\frac1{2b-1}.
\]
\end{example}

\section{Branch-invariant extreme-value and Archimax copulas}
\label{56}
 
Throughout this section, \(K\) denotes a non-empty finite subset of
\(\T\). For \(x=(x_v)_{v\in K}\in[0,\infty)^K\), define the stable
tail dependence function
\begin{equation}
  \ell_K(x)
  =
  \E*{
    \max_{v\in K}x_vZ_v^\alpha
  }.
  \label{57}
\end{equation}
It is finite, continuous, \(1\)-homogeneous, and
\begin{equation}
  \Pp\{\eta_v^{-\alpha}\geq y_v,\ v\in K\}
  =
  \exp\{-\ell_K(y)\},
  \qquad y\in[0,\infty)^K.
  \label{58}
\end{equation}
Consequently, the field
\[
  U_v=\exp\{-\eta_v^{-\alpha}\},
  \qquad v\in\T
\]
has standard uniform margins, and the copula of
\((U_v)_{v\in K}\) is
\begin{equation}
  C_K(u)
  =
  \exp\{-\ell_K(-\log u)\},
  \qquad u\in(0,1]^K.
  \label{59}
\end{equation}
The expression extends continuously to \([0,1]^K\).
 
The families \((\ell_K)\) and \((C_K)\) are projectively consistent:
if \(K\subset L\), then
\[
  \ell_L(x,0_{L\setminus K})=\ell_K(x),
  \qquad
  C_L(u,1_{L\setminus K})=C_K(u).
\]
 
For \(a\in\T\), write
\[
  aK=\{av:v\in K\},
  \qquad
  x^a_{av}=x_v,
  \qquad
  u^a_{av}=u_v
\]
and let  
\[
  \psi:[0,\infty)\to(0,1]
\]
be a completely monotone Archimedean generator satisfying
\[
  \psi(0)=1,
  \qquad
  \lim_{t\to\infty}\psi(t)=0.
\]
Then \(\psi\) is strictly decreasing and has an inverse
\(\psi^{-1}:(0,1]\to[0,\infty)\); see
\cite{charpentier2014multivariate,genest2024copula}. By Bernstein's
theorem, there exists a non-negative random variable \(S_\psi\) such
that
\[
  \E{e^{-tS_\psi}}=\psi(t),
  \qquad t\geq0.
\]
Moreover, we have 
\[
  \Pp\{S_\psi=0\}
  =
  \lim_{t\to\infty}\psi(t)
  =
  0,
\]
and hence \(S_\psi>0\) almost surely. We take \(S_\psi\) independent
of \(\eta\) and set
\begin{equation}
  C_K^\psi(u)
  =
  \psi\left(
    \ell_K\bigl(\psi^{-1}(u)\bigr)
  \right),
  \qquad u\in(0,1]^K,
  \label{60}
\end{equation}
where vector operations are understood componentwise.
 
\begin{proposition}
\label{61}
The field
\begin{equation}
  U_v^\psi
  =
  \psi\left(
    \frac{\eta_v^{-\alpha}}{S_\psi}
  \right),
  \qquad v\in\T,
  \label{62}
\end{equation}
has standard uniform margins and finite-dimensional copulas
\((C_K^\psi)\). In particular, these copulas extend continuously to
\([0,1]^K\) and form a projectively consistent family of Archimax
copulas.
 
The following statements are equivalent:
\begin{enumerate}[(i)]
\item \(\eta\) is branch-stationary;
\item for every non-empty finite \(K\subset\T\),
\begin{equation}
  \ell_{aK}(x^a)=\ell_K(x),
  \qquad
  a\in\T,\quad x\in[0,\infty)^K;
  \label{63}
\end{equation}
\item for every non-empty finite \(K\subset\T\),
\[
  C_{aK}(u^a)=C_K(u),
  \qquad
  a\in\T,\quad u\in(0,1]^K;
\]
\item for every non-empty finite \(K\subset\T\),
\[
  C_{aK}^\psi(u^a)=C_K^\psi(u),
  \qquad
  a\in\T,\quad u\in(0,1]^K;
\]
\item \(U^\psi\) is branch-stationary.
\end{enumerate}
 
If \(Z\) satisfies \eqref{17} and
\(\Theta=\Theta^{[\varnothing]}\), then
\begin{equation}
  \ell_K(x)
  =
  \E*{
    \max_{v\in K}x_v\Theta_v^\alpha
  },
  \qquad x\in[0,\infty)^K.
  \label{64}
\end{equation}
\end{proposition}
 
For \(x\in(0,\infty)^K\), \(f\in\F\), and \(v\in K\), let
$   
  J_K(x,f)
  =
  \operatorname*{arg\,max}_{w\in K}
  x_wf_w^\alpha
$ 
and define
\begin{equation}
  \omega_{K,v}(x,f)
  =
  \frac{\one_{\{v\in J_K(x,f)\}}}{|J_K(x,f)|}.
  \label{65}
\end{equation}
The maps \(\omega_{K,v}\) are Borel measurable and satisfy
\[
  \sum_{v\in K}\omega_{K,v}(x,f)=1,
\]
\[
  \omega_{K,v}(cx,df)=\omega_{K,v}(x,f),
  \qquad c,d>0,
\]
and
\[
  \omega_{aK,av}(x^a,f)
  =
  \omega_{K,v}(x,B^af),
  \qquad a\in\T.
\]
 
\begin{corollary}
\label{66}
For \(v\in K\), define
\begin{equation}
  \Psi_{K,v}(x)
  =
  \E*{
    Z_v^\alpha\omega_{K,v}(x,Z)
  },
  \qquad x\in(0,\infty)^K.
  \label{67}
\end{equation}
Then
\[
  0
  \leq
  \Psi_{K,v}(x)
  =
  \E*{
    \omega_{K,v}(x,\Theta^{[v]})
  }
  \leq1,
  \qquad
  \Psi_{K,v}(cx)=\Psi_{K,v}(x),
  \quad c>0,
\]
and
\begin{equation}
  \ell_K(x)
  =
  \sum_{v\in K}x_v\Psi_{K,v}(x).
  \label{68}
\end{equation}
The field \(\eta\) is branch-stationary if and only if
\begin{equation}
  \Psi_{aK,av}(x^a)
  =
  \Psi_{K,v}(x),
  \qquad
  a\in\T,\quad
  v\in K,\quad
  x\in(0,\infty)^K,
  \label{69}
\end{equation}
for every non-empty finite \(K\subset\T\). If \(Z\) satisfies \eqref{17} and
\(\Theta=\Theta^{[\varnothing]}\), then
\begin{equation}
  \Psi_{K,v}(x)
  =
  \E*{
    \Theta_v^\alpha
    \omega_{K,v}(x,\Theta)
  }.
  \label{70}
\end{equation}
\end{corollary}
 
For a given \(x\in(0,\infty)^K\), suppose that the maximum of
\[
  \left(
    x_w(\Theta_w^{[v]})^\alpha
  \right)_{w\in K}
\]
is almost surely unique. Since \(\Theta_v^{[v]}=1\) almost surely,
\[
  \Psi_{K,v}(x)
  =
  \Pp\left\{
    x_v>
    x_w(\Theta_w^{[v]})^\alpha,
    \quad w\in K\setminus\{v\}
  \right\}.
\]
Thus \(\Psi_{K,v}(x)\) is the probability that \(v\) is the winner
under the \(v\)-rooted spectral law.
 
The decomposition also gives, for \(u\in(0,1)^K\)
\[
  C_K(u)
  =
  \prod_{v\in K}
  u_v^{\Psi_{K,v}(-\log u)},
  \qquad
  C_K^\psi(u)
  =
  \psi\left(
    \sum_{v\in K}
    \psi^{-1}(u_v)
    \Psi_{K,v}\bigl(\psi^{-1}(u)\bigr)
  \right).
\]
 
For \(0<\beta\leq1\), let
\[
  \psi_\beta(t)=e^{-t^\beta},
  \qquad t\geq0,
\]
and let \(S_\beta\) be a positive \(\beta\)-stable random variable,
independent of \(\eta\), satisfying
\[
  \E{e^{-tS_\beta}}=e^{-t^\beta},
  \qquad t\geq0.
\]
For \(\beta=1\), this means \(S_1=1\) almost surely. Setting
\begin{equation}
  \eta_v^{(\beta)}
  =
  S_\beta^{\beta/\alpha}\eta_v^\beta,
  \qquad v\in\T,
  \label{71}
\end{equation}
we have
\[
\begin{aligned}
  \exp\left\{
    -(\eta_v^{(\beta)})^{-\alpha}
  \right\}
  &=
  \exp\left\{
    -\left(
      \frac{\eta_v^{-\alpha}}{S_\beta}
    \right)^\beta
  \right\} =
  \psi_\beta\left(
    \frac{\eta_v^{-\alpha}}{S_\beta}
  \right).
\end{aligned}
\]
By \Cref{61},
\(\eta^{(\beta)}\) has unit \(\alpha\)-Fr\'echet margins and copula
\(C_K^{\psi_\beta}\).
 
For \(x\in[0,\infty)^K\),
\[
\begin{aligned}
  C_K^{\psi_\beta}(e^{-x})
  &=
  \exp\left\{
    -\left[
      \ell_K\bigl((x_v^{1/\beta})_{v\in K}\bigr)
    \right]^\beta
  \right\} \\
  &=
  \exp\{-\ell_K^{(\beta)}(x)\},
\end{aligned}
\]
where
\begin{equation}
  \ell_K^{(\beta)}(x)
  =
  \left[
    \ell_K\bigl((x_v^{1/\beta})_{v\in K}\bigr)
  \right]^\beta.
  \label{72}
\end{equation}
The function \(\ell_K^{(\beta)}\) is \(1\)-homogeneous, since
\[
\begin{aligned}
  \ell_K^{(\beta)}(cx)
  &=
  \left[
    \ell_K\bigl(
      c^{1/\beta}(x_v^{1/\beta})_{v\in K}
    \bigr)
  \right]^\beta=
  c\,\ell_K^{(\beta)}(x).
\end{aligned}
\]
Since \(C_K^{\psi_\beta}\) is a copula and
\(\ell_K^{(\beta)}\) is \(1\)-homogeneous, it is an extreme-value
copula with stable tail dependence function
\(\ell_K^{(\beta)}\). Consequently, \(\eta^{(\beta)}\) is max-stable.
 
If \(\eta\) is branch-stationary, then so is \(\eta^{(\beta)}\),
either directly from \eqref{71} or from the
common-prefix invariance of \(\ell_K^{(\beta)}\). For \(\beta<1\),
this gives a positive-stable extension of the branch-BR models
considered above.

\section{Proofs}
 
\begin{proof}[Proof of \Cref{10}]
Let \(G\in\Hh_\alpha\). Since \(G(\zero)=0\), the set
\[
  A_G=\{f\in\F^\circ:G(f)>1\}
\]
belongs to \(\mathcal B(\F^\circ)\). By the
\(\alpha\)-homogeneity of \(G\) and Tonelli's theorem,
\[
\begin{aligned}
  \nu_X^\circ(A_G)
  &=
  \int_0^\infty
  \Pp\{G(rX)>1\}
  \alpha r^{-\alpha-1}\dd r \\
  &=
  \E*{
    \int_0^\infty
    \one_{\{r^\alpha G(X)>1\}}
    \alpha r^{-\alpha-1}\dd r
  }
  =
  \E{G(X)}.
\end{aligned}
\]
Indeed, for every \(t\in[0,\infty]\),
\[
  \int_0^\infty
  \one_{\{r^\alpha t>1\}}
  \alpha r^{-\alpha-1}\dd r
  =
  t,
\]
with the usual interpretation when \(t=\infty\). Consequently,
\(\nu_X^\circ=\nu_Y^\circ\) implies
\[
  \E{G(X)}=\E{G(Y)},
  \qquad G\in\Hh_\alpha,
\]
and hence \(X\sim_\alpha Y\).
 
Conversely, suppose that \(X\sim_\alpha Y\). For
\(A\in\mathcal B(\F^\circ)\), define
\[
  K_A(x)
  =
  \int_0^\infty
  \one_{\{rx\in A\}}
  \alpha r^{-\alpha-1}\dd r,
  \qquad x\in\F.
\]
The map \(K_A\) is measurable because scalar multiplication
\[
  (r,x)\longmapsto rx
\]
is jointly measurable. Moreover, \(K_A(\zero)=0\), and for every
\(c>0\),
\[
\begin{aligned}
  K_A(cx)
  &=
  \int_0^\infty
  \one_{\{rcx\in A\}}
  \alpha r^{-\alpha-1}\dd r =
  c^\alpha
  \int_0^\infty
  \one_{\{sx\in A\}}
  \alpha s^{-\alpha-1}\dd s
  =
  c^\alpha K_A(x).
\end{aligned}
\]
Thus \(K_A\in\Hh_\alpha\). By Tonelli's theorem,
\[
  \E{K_A(X)}
  =
  \int_0^\infty
  \Pp\{rX\in A\}
  \alpha r^{-\alpha-1}\dd r
  =
  \nu_X^\circ(A),
\]
and similarly
\[
  \E{K_A(Y)}=\nu_Y^\circ(A).
\]
Since \(X\sim_\alpha Y\),
\[
  \nu_X^\circ(A)=\nu_Y^\circ(A),
  \qquad A\in\mathcal B(\F^\circ),
\]
which proves \(\nu_X^\circ=\nu_Y^\circ\).
\end{proof}
 
\begin{proof}[Proof of \Cref{13}]
Fix \(a\in\T\) and put
\[
  \mu_a=\nu_{B^aZ}^\circ,
  \qquad
  \mu=\nu_Z^\circ.
\]
 
For every non-empty finite \(K\subset\T\) and
\(x=(x_v)_{v\in K}\in(0,\infty)^K\), let
\[
  U_{K,x}
  =
  \left\{
    f\in\F^\circ:
    \max_{v\in K}\frac{f_v}{x_v}>1
  \right\}.
\]
Radial integration gives, for every non-negative field \(X\),
\[
  \nu_X^\circ(U_{K,x})
  =
  \E*{
    \max_{v\in K}
    \frac{X_v^\alpha}{x_v^\alpha}
  }.
\]
Since \(B^a\eta\) is represented by \(B^aZ\), it follows that
\(B^a\eta\eqd\eta\) if and only if
\[
  \mu_a(U_{K,x})=\mu(U_{K,x})
\]
for every \(K\) and \(x\).
 
These equalities determine the measures. Indeed, the sets \(U_{K,x}\)
are finite unions of coordinate exceedance sets. Hence
inclusion--exclusion determines the measures of all finite
intersections
\[
  \bigcap_{j=1}^n
  \{f\in\F^\circ:f_{v_j}>x_j\},
\]
which form a \(\pi\)-system generating
\(\mathcal B(\F^\circ)\). All quantities involved are finite, since
for \(X=Z\) or \(X=B^aZ\),
\[
  \nu_X^\circ(U_{K,x})
  \leq
  \sum_{v\in K}x_v^{-\alpha}\E{X_v^\alpha}
  =
  \sum_{v\in K}x_v^{-\alpha}.
\]
Moreover, using the countability of \(\T\),
\[
  \F^\circ
  =
  \bigcup_{v\in\T}\bigcup_{m\geq1}
  \{f\in\F^\circ:f_v>1/m\},
\]
and
\[
  \mu_a\{f:f_v>1/m\}
  =
  \mu\{f:f_v>1/m\}
  =
  m^\alpha.
\]
The uniqueness theorem for \(\sigma\)-finite measures therefore yields
\[
  B^a\eta\eqd\eta
  \quad\Longleftrightarrow\quad
  \mu_a=\mu.
\]
By \Cref{10} we have 
\[
  \mu_a=\mu
  \quad\Longleftrightarrow\quad
  B^aZ\sim_\alpha Z
\]
Furthermore, for \(A\in\mathcal B(\F^\circ)\), 
$
  \mu_a(A)
  =
  \mu\bigl((B^a)^{-1}A\bigr)$ because \(B^a(rf)=rB^af\). Thus \emph{(i)}--\emph{(iii)} are
equivalent.
 
We next compare the local spectral tail fields. Let \(X\) be a
non-negative field with \(\E{X_u^\alpha}=1\), let
\(\Theta_X^{[u]}\) be its local spectral tail field, and put
\[
  E_u=\{f\in\F^\circ:f_u>0\}.
\]
A radial change of variables gives
\begin{equation}
  \nu_X^\circ(A\cap E_u)
  =
  \int_0^\infty
  \Pp\{r\Theta_X^{[u]}\in A\}
  \alpha r^{-\alpha-1}\dd r.
  \label{73}
\end{equation}
In particular, for every \(C\in\mathcal B(\F)\),
\begin{equation}
  \Pp\{\Theta_X^{[u]}\in C\}
  =
  \nu_X^\circ
  \left\{
    f\in\F^\circ:
    f_u>1,\quad f/f_u\in C
  \right\}.
  \label{74}
\end{equation}
 
The local field of \(B^aZ\) at \(u\) satisfies
\[
  \Theta_{B^aZ}^{[u]}
  \eqd
  B^a\Theta^{[au]},
\]
as follows directly from its defining tilt. Hence
\(\mu_a=\mu\), together with
\eqref{74}, implies
\[
  B^a\Theta^{[au]}\eqd\Theta^{[u]},
  \qquad u\in\T.
\]
Thus \emph{(iii)} implies \emph{(iv)}.
 
Conversely, assume \emph{(iv)}. By
\eqref{73},
\[
  \mu_a(A\cap E_u)=\mu(A\cap E_u),
  \qquad
  A\in\mathcal B(\F^\circ),\quad u\in\T.
\]
Since the sets \(E_u\), \(u\in\T\), form a countable cover of
\(\F^\circ\), disjointifying this cover and using countable additivity
gives \(\mu_a=\mu\). Thus \emph{(iv)} implies \emph{(iii)}. Finally, for \(F\in\Hh_0\), zero-homogeneity and the definition of the
local fields give
\[
\begin{aligned}
  \E{F(B^a\Theta^{[au]})}
  &=
  \E*{
    \one_{\{Z_{au}>0\}}
    Z_{au}^\alpha F(B^aZ)
  },\quad
  \E{F(\Theta^{[u]})}
  &=
  \E*{
    \one_{\{Z_u>0\}}
    Z_u^\alpha F(Z)
  }
\end{aligned}
\]
and hence \emph{(iv)} implies \emph{(v)}.
 
Conversely, assume \emph{(v)}. For fixed \(a,u\in\T\) and bounded
non-negative measurable \(H:\F\to[0,\infty)\), define
\[
  F_{u,H}(f)
  =
  \begin{cases}
    H(f/f_u), & f_u>0,\\
    0, & f_u=0.
  \end{cases}
\]
Then \(F_{u,H}\in\Hh_0\), and
\eqref{16} yields
$
  \E{H(B^a\Theta^{[au]})}
  =
  \E{H(\Theta^{[u]})}$ and therefore
\[
  B^a\Theta^{[au]}
  \eqd
  \Theta^{[u]},
\]
so \emph{(v)} implies \emph{(iv)}. Since \(a\in\T\) was arbitrary, all five statements are equivalent.
\end{proof}
 
\begin{proof}[Proof of \Cref{20}]
By \eqref{19}, we have 
\[
  \E{H(B^a\Theta^{[a]})}
  =
  \E*{
    \one_{\{\Theta_a>0\}}
    \Theta_a^\alpha
    H\left(\frac{B^a\Theta}{\Theta_a}\right)
  }.
\]
If \(Z\) is branch-BR stationary, then
$
  B^a\Theta^{[a]}\eqd\Theta
$ is a consequence of \Cref{13}, and hence
\eqref{21} follows.
 
Conversely, assume \eqref{21}. Since \(\Theta\)
is an \(\alpha\)-spectral representer, taking
\(H(f)=f_v^\alpha\) gives
\[
  \E*{
    \one_{\{\Theta_a>0\}}
    \Theta_{av}^\alpha
  }
  =
  \E{\Theta_v^\alpha}
  =
  1
  =
  \E{\Theta_{av}^\alpha}.
\]
Therefore
\[
  \Pp\{\Theta_a=0,\ \Theta_{av}>0\}=0,
  \qquad a,v\in\T.
\]
By countability of \(\T\),
\[
  B^a\Theta=\zero
  \quad\text{on }\{\Theta_a=0\}
  \quad\text{almost surely}.
\]
Consequently, for every \(F\in\Hh_\alpha\),
\[
\begin{aligned}
  \E{F(B^a\Theta)}
  &=
  \E*{
    \one_{\{\Theta_a>0\}}
    \Theta_a^\alpha
    F\left(\frac{B^a\Theta}{\Theta_a}\right)
  } =
  \E{F(\Theta)}.
\end{aligned}
\]
Thus
\[
  B^a\Theta\sim_\alpha\Theta,
  \qquad a\in\T.
\]
Since \(\Theta\) is a representer of the same max-stable field as
\(Z\), \Cref{13} yields branch-BR
stationarity.
 
The first probability in \eqref{22} is zero by
the preceding argument. Moreover, the defining root tilt gives
\[
  \Pp\{\Theta_u=0,\ \Theta_{uv}>0\}
  =
  \E*{
    Z_\varnothing^\alpha
    \one_{\{Z_u=0,\ Z_{uv}>0\}}
  }.
\]
Since \(Z_\varnothing>0\) almost surely, the right-hand side vanishes
if and only if
\[
  \Pp\{Z_u=0,\ Z_{uv}>0\}=0.
\]
This proves \eqref{22}.
\end{proof}
 
\begin{proof}[Proof of \Cref{25}]
For \(u\in\T\), put
\[
  X_v^{(u)}=W_{uv}-W_u,
  \qquad v\in\T.
\]
By \eqref{24} and
\Cref{20}, \emph{(i)} is equivalent to equality
in law of the normalized lognormal fields generated by
\(X^{(u)}\) and \(X^{(\varnothing)}\), for every \(u\in\T\).
Taking logarithms, and using that these Gaussian fields are centred,
shows that this is equivalent to \emph{(ii)}.
 
If \emph{(ii)} holds, then
\[
\begin{aligned}
  \Gamma(uv,uw)
  &=
  \operatorname{Var}
  \left(
    (W_{uv}-W_u)-(W_{uw}-W_u)
  \right) \\
  &=
  \operatorname{Var}
  \left(
    (W_v-W_\varnothing)-(W_w-W_\varnothing)
  \right)
  =
  \Gamma(v,w),
\end{aligned}
\]
so \emph{(iii)} follows.
 
Conversely, under \eqref{26},
\[
\begin{aligned}
  \operatorname{Cov}
  (W_{uv}-W_u,W_{uw}-W_u)
  &=
  \frac12
  \left\{
    \Gamma(uv,u)
    +
    \Gamma(uw,u)
    -
    \Gamma(uv,uw)
  \right\} \\
  &=
  \frac12
  \left\{
    \Gamma(v,\varnothing)
    +
    \Gamma(w,\varnothing)
    -
    \Gamma(v,w)
  \right\}.
\end{aligned}
\]
This is the covariance of
\((W_v-W_\varnothing,W_w-W_\varnothing)\). Since the increment fields
are centred Gaussian, \emph{(ii)} follows.
\end{proof}
 
\begin{proof}[Proof of \Cref{39}]
Under \eqref{17},
\(Z_\varnothing^\alpha>0\) almost surely, so the tilted measure
\(Z_\varnothing^\alpha\dd\Pp\) and \(\Pp\) have the same null sets.
Since multiplication by the positive finite scalar
\(Z_\varnothing^{-1}\) does not affect finiteness of the forward
\(\alpha\)-mass, the defining root tilt
\eqref{11} gives the equivalence of
\emph{(i)} and \emph{(ii)}.
 
The equivalence of \emph{(ii)} and \emph{(iii)} follows from
\[
  0<R_{\mathrm P}<\infty
  \qquad\text{almost surely}.
\]
Finally, conditional on \(\Theta\), Tonelli's theorem and the Pareto
tail implies 
\[
\begin{aligned}
  \E{N^+(Y)\mid\Theta}
  &=
  \sum_{v\in\T}
  \Pp\{R_{\mathrm P}\Theta_v>1\mid\Theta\} =
  \sum_{v\in\T}
  \min\{1,\Theta_v^\alpha\}.
\end{aligned}
\]
For every non-negative sequence \((x_v)\)
\[
  \sum_v\min\{1,x_v\}<\infty
  \quad\Longleftrightarrow\quad
  \sum_vx_v<\infty.
\]
This proves the equivalence of \emph{(ii)} and \emph{(iv)}. The same
arguments with finiteness replaced by infiniteness prove the final
assertion.
\end{proof}
 
\begin{proof}[Proof of \Cref{43}]
Tonelli's theorem and the bijectivity of \(g\mapsto gv\) give
\[
\begin{aligned}
  \sum_{g\in\G}
  \E{\rho_\pi(Q^{[g]})^\alpha}
  &=
  \sum_{v\in\T}\pi_v
  \sum_{g\in\G}\E{Q_{gv}^\alpha} =
  \sum_{v\in\T}\pi_v
  \sum_{h\in\G}\E{Q_h^\alpha}
  =
  c_Q.
\end{aligned}
\]
Thus \eqref{44} defines a probability law. Taking
\[
  H(f)=\one_{\{\rho_\pi(f)=1\}}
\]
in \eqref{44} yields 
\[
  \rho_\pi(Z^{Q,\pi})=1
  \qquad\text{almost surely}.
\]
Similarly, for \(w\in\T\),
\[
  \Ehat{(Z_w^{Q,\pi})^\alpha}
  =
  \frac1{c_Q}
  \sum_{g\in\G}\E{Q_{gw}^\alpha}
  =
  1.
\]
Moreover, for each \(g\in\G\), almost surely on
\(\{\rho_\pi(Q^{[g]})>0\}\)
\[
\begin{aligned}
  \sum_{v\in\T}
  \left(
    \frac{Q_{gv}}{\rho_\pi(Q^{[g]})}
  \right)^\alpha
  &\leq
  \frac{\sum_{h\in\G}Q_h^\alpha}
       {\rho_\pi(Q^{[g]})^\alpha}
  <
  \infty.
\end{aligned}
\]
The mixture in \eqref{44} is therefore supported on
\(\alpha\)-summable fields.
 
If \(F\in\Hh_\alpha\), homogeneity cancels the gauge in
\eqref{44} and yields
\eqref{45}. Hence, for \(u\in\T\)
\[
\begin{aligned}
  \Ehat{F(B^uZ^{Q,\pi})}
  &=
  \frac1{c_Q}
  \sum_{g\in\G}\E{F(Q^{[gu]})} =
  \frac1{c_Q}
  \sum_{h\in\G}\E{F(Q^{[h]})}
  =
  \Ehat{F(Z^{Q,\pi})},
\end{aligned}
\]
because \(g\mapsto gu\) is a bijection. This proves branch-BR
stationarity and shows that the homogeneous class is independent of
\(\pi\).
 
On \(\{\rho_\pi(Q^{[g]})>0\}\), the contribution of the \(g\)-th
mixture component to the tail measure is
\[
  \rho_\pi(Q^{[g]})^\alpha
  \int_0^\infty
  \one_{\{rQ^{[g]}/\rho_\pi(Q^{[g]})\in A\}}
  \alpha r^{-\alpha-1}\dd r.
\]
The substitution \(s=r/\rho_\pi(Q^{[g]})\) cancels the gauge factor
and gives \eqref{46}.
 
For \(u\in\T\) and \(A\in\mathcal B(\F^\circ)\),
\[
\begin{aligned}
  \nu_Q\bigl((B^u)^{-1}A\bigr)
  &=
  \frac1{c_Q}
  \sum_{g\in\G}
  \int_0^\infty
  \Pp\left\{
    rQ^{[gu]}\in A,\ Q^{[g]}\neq\zero
  \right\}
  \alpha r^{-\alpha-1}\dd r.
\end{aligned}
\]
Since \(A\subset\F^\circ\), the event
\(\{rQ^{[gu]}\in A\}\) implies \(Q^{[gu]}\neq\zero\), and hence also
\(Q^{[g]}\neq\zero\). Therefore
\[
\begin{aligned}
  \nu_Q\bigl((B^u)^{-1}A\bigr)
  &=
  \frac1{c_Q}
  \sum_{g\in\G}
  \int_0^\infty
  \Pp\{rQ^{[gu]}\in A\}
  \alpha r^{-\alpha-1}\dd r =
  \frac1{c_Q}
  \sum_{h\in\G}
  \int_0^\infty
  \Pp\{rQ^{[h]}\in A\}
  \alpha r^{-\alpha-1}\dd r =
  \nu_Q(A),
\end{aligned}
\]
where the second equality uses the bijection \(g\mapsto gu\).
 
Moreover,
\[
  \nu_Q\{f:f_w>1/m\}=m^\alpha,
  \qquad w\in\T,\quad m\geq1.
\]
Since \(\T\) is countable, these sets form a countable
finite-measure cover of \(\F^\circ\). Together with
\(\nu_Q\{\zero\}=0\), this shows that \(\nu_Q\) is
\(\sigma\)-finite on \(\F\).
 
Finally, tilting \eqref{44} by the root coordinate
and normalizing at the root gives
\[
\begin{aligned}
  \E{H(\Theta^Q)}
  &=
  \Ehat*{
    \one_{\{Z_\varnothing^{Q,\pi}>0\}}
    (Z_\varnothing^{Q,\pi})^\alpha
    H\left(
      \frac{Z^{Q,\pi}}{Z_\varnothing^{Q,\pi}}
    \right)
  } =
  \frac1{c_Q}
  \sum_{g\in\G}
  \E*{
    \one_{\{Q_g>0\}}
    Q_g^\alpha
    H\left(
      \left(
        \frac{Q_{gv}}{Q_g}
      \right)_{v\in\T}
    \right)
  }.
\end{aligned}
\]
This proves \eqref{48} and shows that its law
does not depend on \(\pi\). On \(\{Q_g>0\}\),
\[
  \sum_{v\in\T}
  \left(
    \frac{Q_{gv}}{Q_g}
  \right)^\alpha
  \leq
  \frac{\sum_{h\in\G}Q_h^\alpha}{Q_g^\alpha}
  <
  \infty.
\]
This proves the last assertion.
\end{proof}
 
\begin{proof}[Proof of \Cref{49}]
The gauge-biased law gives
\[
  \widehat{\Pp}_\pi
  \{Z_\varnothing^{Q,\pi}=0\}
  =
  \frac1{c_Q}
  \sum_{g\in\G}
  \E*{
    \rho_\pi(Q^{[g]})^\alpha
    \one_{\{Q_g=0\}}
  }.
\]
Because every \(\pi_v\) is positive,
\(\rho_\pi(Q^{[g]})>0\) if and only if
\(Q_{gv}>0\) for at least one \(v\in\T\). Since \(\G\) and \(\T\)
are countable, the displayed probability vanishes if and only if
\eqref{50} holds.
\end{proof}
 
\begin{proof}[Proof of \Cref{40}]
Let \(\Phi\) be the standard normal distribution function, fix
\(p\in(0,1)\), and put
\[
  c=\Phi^{-1}(p).
\]
Take \(\mathcal A=\{2,3,\ldots\}\) and define
\[
  a_k=c+\sqrt{c^2+2\log k},
  \qquad k\geq2.
\]
Then \(a_k>0\) and
\begin{equation}
  ca_k-\frac{a_k^2}{2}=-\log k.
  \label{75}
\end{equation}
For \(v=(k_1,\ldots,k_n)\), let
\[
  A_v=\sum_{j=1}^n a_{k_j},
  \qquad
  A_\varnothing=0.
\]
If \(X\) is standard normal, set
\begin{equation}
  W_v=A_vX,
  \qquad
  Z_v
  =
  \exp\left\{
    A_vX-\frac{A_v^2}{2}
  \right\}.
  \label{76}
\end{equation}
Here \(W=(W_v)_{v\in\T}\) is a rank-one centered Gaussian field, and
\(Z\) is its associated lognormal spectral field. Moreover,
\(Z_\varnothing=1\), \(\E{Z_v}=1\), and
\[
  A_{uv}=A_u+A_v.
\]
Consequently, we have 
\[
  \Gamma(uv,uw)
  =
  (A_{uv}-A_{uw})^2
  =
  (A_v-A_w)^2
  =
  \Gamma(v,w).
\]
Thus \Cref{25} shows that \(Z\) is branch-BR
stationary.
 
Conditionally on \(X=x\), the mass of the first generation is
\begin{equation}
\begin{aligned}
  F(x)
  &:=
  \sum_{k\geq2}
  \exp\left\{
    xa_k-\frac{a_k^2}{2}
  \right\} =
  \sum_{k\geq2}
  \frac1k
  \exp\{(x-c)a_k\}.
\end{aligned}
  \label{77}
\end{equation}
Since \(a_k\sim\sqrt{2\log k}\), the integral test with the
substitution \(s=\sqrt{2\log t}\) gives
\begin{equation}
  F(x)<\infty
  \quad\Longleftrightarrow\quad
  x<c.
  \label{78}
\end{equation}
For \(x\geq c\), the series in
\eqref{77} dominates the harmonic series.
 
We next show that convergence of the first-generation mass implies
convergence over the whole tree. Put
\[
  a_*=a_2>0.
\]
If \(v=(k_1,\ldots,k_n)\), then
\[
\begin{aligned}
  \frac{A_v^2}{2}
  &=
  \frac12\sum_{j=1}^n a_{k_j}^2
  +
  \sum_{1\leq i<j\leq n}a_{k_i}a_{k_j} \geq
  \frac12\sum_{j=1}^n a_{k_j}^2
  +
  \binom n2a_*^2.
\end{aligned}
\]
Hence for \(x<c\)
\[
\begin{aligned}
  \sum_{v\in\mathcal A^n}Z_v
  &\leq
  \exp\left\{
    -\binom n2a_*^2
  \right\}
  \sum_{(k_1,\ldots,k_n)\in\mathcal A^n}
  \prod_{j=1}^n
  \exp\left\{
    xa_{k_j}-\frac{a_{k_j}^2}{2}
  \right\} =
  \exp\left\{
    -\binom n2a_*^2
  \right\}
  F(x)^n.
\end{aligned}
\]
The right-hand side is summable over \(n\). For \(x\geq c\), the first
generation already has infinite mass. Therefore
 we have $
  \{S^+(Z)<\infty\}
  =
  \{X<c\},
$
and hence
$
  \Pp\{S^+(Z)<\infty\}
  =
  \Phi(c)
  =
  p.$
\end{proof}
 
\begin{proof}[Proof of \Cref{61}]
Since \(\psi\) is strictly decreasing, for \(u\in(0,1]^K\),
\[
  U_v^\psi\leq u_v
  \quad\Longleftrightarrow\quad
  \eta_v^{-\alpha}
  \geq
  S_\psi\psi^{-1}(u_v).
\]
Conditioning on \(S_\psi\), using
\eqref{58} and the homogeneity of \(\ell_K\),
we obtain
\[
\begin{aligned}
  \Pp\{U_v^\psi\leq u_v,\ v\in K\mid S_\psi\}
  &=
  \exp\left\{
    -\ell_K\bigl(S_\psi\psi^{-1}(u)\bigr)
  \right\} =
  \exp\left\{
    -S_\psi\ell_K\bigl(\psi^{-1}(u)\bigr)
  \right\}.
\end{aligned}
\]
Taking expectations gives \eqref{60}. For
\(K=\{v\}\), since \(\ell_{\{v\}}(t)=t\),
\[
  \Pp\{U_v^\psi\leq u\}
  =
  \psi(\psi^{-1}(u))
  =
  u.
\]
Projective consistency follows from
\[
  \ell_L(x,0_{L\setminus K})=\ell_K(x)
  \qquad\text{and}\qquad
  \psi^{-1}(1)=0.
\]
 
Under the canonical identification of \(aK\) with \(K\), the stable
tail dependence function of \((\eta_{av})_{v\in K}\) is
\[
  x\longmapsto\ell_{aK}(x^a).
\]
Hence \emph{(i)} and \emph{(ii)} are equivalent when the identities
are required for every non-empty finite \(K\subset\T\). Formula
\eqref{59} gives the equivalence of
\emph{(ii)} and \emph{(iii)}.
 
Condition \emph{(ii)} implies \emph{(iv)} by
\eqref{60}. Conversely, taking
\[
  u_v=\psi(x_v),
  \qquad v\in K,
\]
and using the injectivity of \(\psi\), condition \emph{(iv)} yields
$  
  \ell_{aK}(x^a)=\ell_K(x).$
Finally, since \(U^\psi\) has standard uniform margins, invariance of
its finite-dimensional copulas is equivalent to branch stationarity.
Thus \emph{(iv)} and \emph{(v)} are equivalent.
 
Under \eqref{17}, we have \(Z_\varnothing>0\) almost surely. The
root tilt therefore gives
\[
\begin{aligned}
  \E*{
    \max_{v\in K}x_v\Theta_v^\alpha
  }
  &=
  \E*{
    Z_\varnothing^\alpha
    \max_{v\in K}
    x_v
    \left(
      \frac{Z_v}{Z_\varnothing}
    \right)^\alpha
  } =
  \E*{
    \max_{v\in K}x_vZ_v^\alpha
  }
  =
  \ell_K(x),
\end{aligned}
\]
which proves \eqref{64}.
\end{proof}
 
\begin{proof}[Proof of \Cref{66}]
For every \(x\in(0,\infty)^K\) and \(f\in\F\),
\[
  \sum_{v\in K}
  x_vf_v^\alpha\omega_{K,v}(x,f)
  =
  \max_{v\in K}x_vf_v^\alpha.
\]
Taking expectations gives \eqref{68}. The bound
and homogeneity of \(\Psi_{K,v}\) follow directly from
\[
  0\leq\omega_{K,v}\leq1
  \qquad\text{and}\qquad
  \E{Z_v^\alpha}=1.
\]
 
Since \(\omega_{K,v}(x,\cdot)\) is zero-homogeneous, the definition of
the local spectral tail field gives
\[
\begin{aligned}
  \E*{
    \omega_{K,v}(x,\Theta^{[v]})
  }
  &=
  \E*{
    \one_{\{Z_v>0\}}
    Z_v^\alpha
    \omega_{K,v}\left(
      x,\frac{Z}{Z_v}
    \right)
  } =
  \E*{
    Z_v^\alpha\omega_{K,v}(x,Z)
  }
  =
  \Psi_{K,v}(x).
\end{aligned}
\]
 
Suppose that \(\eta\) is branch-stationary. By
\Cref{13},
\[
  B^a\Theta^{[av]}
  \eqd
  \Theta^{[v]}.
\]
Therefore
\[
\begin{aligned}
  \Psi_{aK,av}(x^a)
  &=
  \E*{
    \omega_{aK,av}(x^a,\Theta^{[av]})
  } =
  \E*{
    \omega_{K,v}(x,B^a\Theta^{[av]})
  } =
  \E*{
    \omega_{K,v}(x,\Theta^{[v]})
  }
  =
  \Psi_{K,v}(x).
\end{aligned}
\]
 
Conversely, suppose that
\eqref{69} holds for every non-empty finite
\(K\subset\T\). Then
\[
\begin{aligned}
  \ell_{aK}(x^a)
  &=
  \sum_{v\in K}
  x_v\Psi_{aK,av}(x^a) =
  \sum_{v\in K}
  x_v\Psi_{K,v}(x)
  =
  \ell_K(x)
\end{aligned}
\]
for \(x\in(0,\infty)^K\). By continuity, this identity extends to
\([0,\infty)^K\), and
\Cref{61} yields branch stationarity. Finally, under \eqref{17}
\[
\begin{aligned}
  \E*{
    \Theta_v^\alpha
    \omega_{K,v}(x,\Theta)
  }
  &=
  \E*{
    Z_\varnothing^\alpha
    \left(
      \frac{Z_v}{Z_\varnothing}
    \right)^\alpha
    \omega_{K,v}\left(
      x,\frac{Z}{Z_\varnothing}
    \right)
  } =
  \E*{
    Z_v^\alpha\omega_{K,v}(x,Z)
  }
  =
  \Psi_{K,v}(x),
\end{aligned}
\]
which proves \eqref{70}.
\end{proof}

\bibliographystyle{ieeetr}
\bibliography{GG_v26}
 
\end{document}